\documentclass{lms-ppn}
\usepackage{amsmath,amsfonts,amssymb,amsxtra,hyperref,color,graphicx}

\newtheorem{thm}{Theorem}[section]
\newtheorem{cor}[thm]{Corollary}
\newtheorem{lem}[thm]{Lemma}
\newtheorem{prop}[thm]{Proposition}

\newnumbered{assertion}{Assertion}
\newnumbered{conjecture}{Conjecture}
\newnumbered{definition}{Definition}
\newnumbered{hypothesis}{Hypothesis}
\newnumbered{note}{Note}
\newnumbered{observation}{Observation}
\newnumbered{problem}{Problem}
\newnumbered{algorithm}{Algorithm}
\newnumbered{example}{Example}
\newunnumbered{notation}{Notation}

\newcommand{\Email}[1]{{\sl E-mail:\/} {\rm\textsf{#1}}}
\newcommand{\R}{{\mathbb R}}
\renewcommand{\S}{{\mathbb S}}

\newcommand{\be}[1]{\begin{equation}\label{#1}}
\newcommand{\ee}{\end{equation}}
\renewcommand{\(}{\left(}
\renewcommand{\)}{\right)}

\renewcommand{\S}{\mathbb{S}}

\newcommand{\iRd}[1]{\int_{\R^d}{#1}\;dx}
\newcommand{\iR}[1]{\int_{\R}{#1}\;dx}
\newcommand{\iRy}[1]{\int_{\R}{#1}\;dy}
\newcommand{\il}[1]{\int_{\R}{#1}\;dx}
\newcommand{\ilcpt}[1]{\int_{-\frac\pi2}^{\frac\pi2}{#1}\;dx}

\newcommand{\nrmR}[2]{\|{#1}\|_{\mathrm L^{#2}(\R^d)}}
\newcommand{\nrml}[2]{\|{#1}\|_{\mathrm L^{#2}(\R)}}
\newcommand{\nrmlcpt}[2]{\|{#1}\|_{\mathrm L^{#2}(I)}}

\renewcommand{\H}{\mathrm H}

\renewcommand{\L}{{\mathcal L}\,}

\newcommand{\izp}[1]{\int_{-1}^1{#1}\;d\nu_p}
\newcommand{\izpp}[1]{\int_{\R}#1\;d\xi_p}

\newcommand{\scal}[2]{\left\langle{#1},{#2}\right\rangle}
\renewcommand{\aa}{\mathsf a}
\newcommand{\bb}{\mathsf b}
\newcommand{\ixmu}[2]{\int_\Omega{#1}\;d\mu_{#2}}
\newcommand{\nub}{d\mu_\bb}
\newcommand{\ixb}[1]{\ixmu{#1}{\bb}}
\newcommand{\Lab}{\mathcal L_{\aa\bb}\,}
\newcommand{\ir}[1]{\int_0^\infty{#1}\;r^{d-1}\,dr}

\newcommand{\newclaim}[1]{\par\medskip\noindent$\bullet$ \emph{#1}\par\smallskip}

\usepackage{color}

\definecolor{darkgreen}{rgb}{0,0.4,0}
\newcommand{\darkgreen}[1]{\color{darkgreen}}

\newcommand{\z}{y}

\title[Interpolation, duality and flows]{One-dimensional Gagliardo-Nirenberg-Sobolev inequalities:\\Remarks on duality and flows}

\author[J.~Dolbeault, M.~Esteban, A.~Laptev \& M.~Loss]{Jean Dolbeault, Maria J.~Esteban, Ari Laptev and Michael Loss}

\classno{26D10 (primary); 46E35; 35K55 (secondary)
}

\extraline{J.D.~has been partially supported by ANR grants \emph{STAB} and \emph{Kibord}. He thanks the School of Mathematics of the Georgia Institute of Technology for welcoming him. J.D.~and M.E.~have been partially supported by the ANR grant \emph{NoNAP}. J.D.~and M.L.~have been supported in part respectively by NSF grant DMS-1301555. A.L.~and M.L.~thank the Institut Henri Poincar\'e.}

\begin{document}
\maketitle

\begin{abstract}
This paper is devoted to one-dimensional interpolation Gagliardo-Nirenberg-Sobolev inequalities. We study how various notions of duality, transport and monotonicity of functionals along flows defined by some nonlinear diffusion equations apply. 

We start by reducing the inequality to a much simpler dual variational problem using mass transportation theory. Our second main result is devoted to the construction of a Lyapunov functional associated with a nonlinear diffusion equation, that provides an alternative proof of the inequality. The key observation is that the inequality on the line is equivalent to Sobolev's inequality on the sphere, at least when the dimension is an integer, or to the critical interpolation inequality for the ultraspherical operator in the general case. The time derivative of the functional along the flow is itself very interesting. It explains the machinery of some rigidity estimates for nonlinear elliptic equations and shows how eigenvalues of a linearized problem enter in the computations. Notions of gradient flows are then discussed for various notions of distances.

Throughout this paper we shall deal with two classes of inequalities corresponding either to $p>2$ or to $1 < p<2$. The algebraic part in the computations is very similar in both cases, although the case $1< p<2$ is definitely less standard.
\end{abstract}

\section{Introduction}\label{Sec:Intro}

When studying sharp functional inequalities, and the corresponding best constants and optimizers, one has essentially three strategies at hand:

(a) To use a \emph{direct variational method} where one establishes the existence of optimizers. Then by analyzing the solutions of the corresponding \emph{Euler-Lagrange equations}, one can sometimes obtain explicit values for the optimizers and for the best constants.

(b) It is an old idea that \emph{flows} on function spaces and sharp functional inequalities are intimately related. Sharp inequalities are used to study qualitative and quantitative properties of flows such as decay rates of the solutions in certain norms. A famous example is Nash's inequality that provides exact decay rates for heat kernels \cite{MR0100158,MR1103113}. Conversely, flows can be used to prove sharp inequalities and identify the optimizers. A famous example is the derivation of the logarithmic Sobolev inequality by D.~Bakry and M.~Emery using the heat flow \cite{MR889476,MR808640,MR772092}. In that case, the flow relates an arbitrary initial datum to an optimizer of the inequality. The monotonicity of an appropriate functional along the flow provides \emph{a priori} estimates that, in case of critical points, can be related with older methods for proving \emph{rigidity results} in nonlinear elliptic equations. See~\cite{dolbeault:hal-00784887} and references therein for more details.

(c) Another way to look at these problems is to use the \emph{mass transportation theory}. One does not transport a function to an optimizer but instead one transports an arbitrary function to another one leading to a new variational problem. This dual variational problem can be easier to deal with than the original one. A well-known example of this method has been given by D.~Cordero-Erausquin, B.~Nazaret and C.~Villani in \cite{MR2032031} (also see \cite{MR2053603} for a simpler proof). With this approach, they obtained proofs of some of the Gagliardo-Nirenberg-Sobolev inequalities.

\medskip In this paper we will focus on another family of one-dimensional Gagliardo-Nirenberg-Sobolev inequalities that can be written as
\begin{eqnarray}
&&\nrml fp\le\mathsf C_{\rm GN}(p)\,\nrml{f'}2^\theta\,\nrml f2^{1-\theta}\quad\mbox{if}\quad p\in(2,\infty)\,,\label{GN1}\\
&&\nrml f2\le\mathsf C_{\rm GN}(p)\,\nrml{f'}2^\eta\,\nrml fp^{1-\eta}\quad\mbox{if}\quad p\in(1,2)\,,\label{GN2}
\end{eqnarray}
with $\theta=\frac{p-2}{2\,p}$ and $\eta=\frac{2-p}{2+p}$. See \cite{MR0109295,MR0102740,MR0109940} for the original papers. The threshold case corresponding to the limit as $p\to2$ is the logarithmic Sobolev inequality
\be{log-Sob}
\il{u^2\,\log\(\frac{u^2}{\nrml u2^2}\)}\le\frac 12\,\nrml u2^2\,\log\(\frac 2{\pi\,e}\,\frac{\nrml{u'}2^2}{\nrml u2^2}\)\,,
\ee
derived in \cite{MR0420249}. 

Among Gagliardo-Nirenberg-Sobolev inequalities, there are only a few cases for which best constants are explicit and optimal functions can be simply characterized. Let us mention Nash's inequality (see \cite{MR1230297}) and some interpolation inequalities on the sphere (see \cite{MR1134481,MR1230930}). A family for which such issues are known is
\[
\nrmR f{2q}\le\mathsf K_{\rm GN}(q,d)\,\nrmR{f'}2^\theta\,\nrmR f{q+1}^{1-\theta}\,,
\]
if $q\in(1,\infty)$ when $d=1$ or $2$, and $q\in(1,\frac d{d-2}]$ when $d\ge3$, and
\[
\nrmR f{q+1}\le\mathsf K_{\rm GN}(q,d)\,\nrmR{f'}2^\theta\,\nrmR f{2q}^{1-\theta}\,,
\]
if $q\in(0,1)$, with appropriate values of $\theta$. See \cite{Gunson91,MR1940370}. Again the logarithmic Sobolev inequality appears as the threshold case corresponding to the limit $q\to1$. These inequalities have two important properties:
\begin{enumerate}
\item[-] There is a nonlinear flow (a fast diffusion flow if $q>1$ and a porous media flow if $q<1$) which is associated to them. This flow can be considered as a gradient flow of an \emph{entropy} functional with respect to Wasserstein's distance, as was noticed in \cite{MR1842429}.
\item[-] A duality argument based on mass transportation methods allows to relate these inequalities with much simpler ones, as was observed in \cite{MR2032031,MR2053603}.
\end{enumerate}
The purpose of our paper is to study the analogue of these properties in case of \eqref{GN1} and~\eqref{GN2}. We will apply the methods described in (a), (b) and (c). Method (a) is rather standard (the proof is given in the appendix for completeness) while (b) and (c), although not extremely complicated, are less straightforward. As far as we know, neither (b) nor (c) have been applied yet to \eqref{GN1} and~\eqref{GN2}. Method (a) relies on compactness arguments, method (b) relies on \emph{a priori} estimates related to a global flow and method (c) requires the existence of a transport map.

\medskip Let us denote by $\mathrm L_2^1(\R)$ the space of the functions $\left\{G\in\mathrm L^1(\R)\,:\,\iRy{G\,|y|^2}<\infty\right\}$ and define
\be{Eqn:cp}
\mathsf c_p:=\left\{\begin{array}{l}
\(\tfrac{p+2}2\)^\frac{2\,(p-2)}{3\,p-2}\quad\mbox{if}\quad p\in(2,\infty)\vspace*{6pt}\cr
2^\frac{2-p}{4-p}\quad\mbox{if}\quad p\in(1,2)
\end{array}\right.
\ee
Based on mass transportation theory, method (c) allows to relate the minimization problem associated with \eqref{GN1} and~\eqref{GN2} to a dual variational problem as follows.
\begin{thm}\label{Thm:duality} The following inequalities hold: if $p\in(2,\infty)$, we have
\be{Ineq:duality1}
\sup_{G\in\mathrm L_2^1(\R)\setminus\{0\}}\frac{\iRy{G^\frac{p+2}{3\,p-2}}}{\(\iRy{G\,|y|^2}\)^\frac{p-2}{3\,p-2}\,\(\iRy{G}\)^\frac4{3\,p-2}}=\mathsf c_p\,\inf_{f\in\H^1(\R)\setminus\{0\}}\frac{\nrml {f'}2^\frac{2\,(p-2)}{3\,p-2}\,\nrml f2^\frac{2\,(p+2)}{3\,p-2}}{\nrml fp^\frac{4\,p}{3\,p-2}}\,,
\ee
and if $p\in(1,2)$, we obtain
\be{Ineq:duality2}
\sup_{G\in\mathrm L_2^1(\R)\setminus\{0\}}\frac{\iRy{G^\frac2{4-p}}}{\(\iRy{G\,|y|^2}\)^\frac{2-p}{2\,(4-p)}\,\(\iRy{G}\)^\frac{p+2}{2\,(4-p)}}=\mathsf c_p\,\inf_{f\in\H^1(\R)\setminus\{0\}}\,\frac{\nrml{f'}2^\frac{2-p}{4-p}\,\nrml fp^\frac{2\,p}{4-p}}{\nrml f2^\frac{p+2}{4-p}}\,.
\ee
\end{thm}
All variational problems in Theorem~\ref{Thm:duality} have explicit extremal functions. The maximization problems is rather straightforward and yields an efficient method for computing $\mathsf C_{\rm GN}(p)$ in both of the cases corresponding to \eqref{Ineq:duality1} and \eqref{Ineq:duality2}. The proof of Theorem~\ref{Thm:duality} will be given in Sections~\ref{Sec:Duality} and~\ref{Sec:Optimality}.

\medskip Next we shall focus on the method (b). In this spirit, let us define on $\H^1(\R)$ the functional
\be{atlanta5}
\mathcal F[v]:=\nrml{v'}2^2+\frac4{(p-2)^2}\,\nrml v2^2-\mathsf C\,\nrml vp^2
\ee
where $\mathsf C$ is such that $\mathcal F[v_\star]=0$, with
\[
v_\star(x):=(\cosh x)^{-\frac2{p-2}}\,.
\]
Notice that $v_\star(x)=\big(1-z(x)^2\big)^\frac1{p-2}$ if $z(x):=\tanh x$, for any $x\in\R$. Next, consider the \emph{flow} associated with the nonlinear evolution equation
\be{Eqn:flowone}
v_t=\frac{v^{1-\frac p2}}{\sqrt{1-z^2}}\left[v''+\,\frac{2\,p}{p-2}\,z\,v'+\,\frac p2\,\frac{|v'|^2}v+\,\frac2{p-2}\,v\,\right]\,.
\ee

Then $\mathcal F$ is monotone non-increasing along the flow defined by \eqref{Eqn:flowone}.
\begin{thm}\label{thm:flow1} Let $p\in(2,\infty)$. Assume that $v_0\in\H^1(\R)$ is positive, such that $\nrml{v_0}p=\nrml{v_\star}p$ and the limits $\lim_{x\to\pm\infty}\frac{v_0(x)}{v_\star(x)}$ exist. If $v$ is a solution of \eqref{Eqn:flowone} with initial datum $v_0$, then we have
\[
\frac d{dt}\mathcal F[v(t)]\le 0\quad\mbox{and}\quad\lim_{t\to\infty}\mathcal F[v(t)]=0\,.
\]
Moreover, $\frac d{dt}\mathcal F[v(t)]=0$ if and only if, for some $x_0\in\R$, $v_0(x)=v_\star(x-x_0)$ for any $x\in\R$.\end{thm}
This result deserves a few comments. First of all by proper scaling it yields a proof of~\eqref{Ineq:duality1}. Further it shows that up to translations, multiplication by a constant and scalings, the function~$v_\star$ is the unique optimal function for \eqref{GN1}, and again allows to compute $\mathsf C_{\rm GN}(p)$. Indeed we have shown that
\begin{equation} \label{atlanta1}
\nrml{v_0'}2^2+\frac4{(p-2)^2}\,\nrml {v_0} 2^2-\mathsf C\,\nrml {v_0}p^2 \ge\lim_{t\to\infty}\mathcal F[v(t)]=0
\end{equation}
for an arbitrary function $v_0$ satisfying the assumptions of Theorem~\ref{thm:flow1}, but the technical conditions on $v_0$ can easily be removed at the level of the inequality by a density argument.

At first sight, \eqref{Eqn:flowone} may look complicated. The interpolation inequality \eqref{Ineq:duality1} turns out to be equivalent to Sobolev's inequality on the $d$-dimensional sphere if $d=\frac{2\,p}{p-2}$ is an integer, and to the critical interpolation inequality for the ultraspherical operator in the general case. These considerations will be detailed in Section~\ref{Sec:monotonicity}.

For completion, let us mention that a \emph{rigidity result} is associated with Theorem~\ref{thm:flow1}. A statement will be given in Section~\ref{Sec:Rigidity}. Although the rigidity result can be obtained directly, the flow approach is simpler to state, at least in the original variables, and provides a clear scheme.

\medskip In the case $1<p<2$, a result similar to Theorem \ref{thm:flow1} can be proved. In that case the global attractor is defined as
\[
v_*(x)=(\cos x)^\frac 2{2-p}\quad\mbox{if}\;x\in I:=\(-\tfrac\pi2,\tfrac\pi2\)\quad\mbox{and}\quad v_*(x)=0\quad\mbox{otherwise}\,.
\]
Then the functional
\be{atlanta6}
\mathcal F[v]:=\nrmlcpt{v'}2^2+\,\mathsf C\,\nrmlcpt vp^2-\,\frac4{(2-p)^2}\,\nrmlcpt v2^2\,,
\ee
defined for any $v\in\H^1(\R)\cap\mathrm L^p(\R)$, where again $\mathsf C$ is chosen such that $\mathcal F[v_\star]=0$, is non-increasing along the flow defined by
\be{Eqn:flowtwo}
v_t=\frac{v^{1-\frac p2}}{\sqrt{1+\z^2}}\left[v''+\,\frac{2\,p}{2-p}\,\z\,v'+\,\frac p2\,\frac{|v'|^2}v+\,\frac2{2-p}\,v\,\right]\,,
\ee
where $\z(x)=\tan x$. More precisely, we have a result that goes exactly as Theorem~\ref{thm:flow1} and, up to necessary adaptations due to the fact that optimal functions are compactly supported, there are similar consequences that we will not list here.
\begin{thm}\label{thm:flow2} Let $p\in(1,2)$. Assume that $v_0\in\H^1(I)$ is positive, such that $\nrmlcpt{v_0}p=\nrmlcpt{v_\star}p$ and the limits $\lim_{x\to\pm\frac\pi2}\frac{v_0(x)}{v_\star(x)}$ exist. If $v$ is a solution of \eqref{Eqn:flowtwo} with initial datum $v_0$, then we have
\[
\frac d{dt}\mathcal F[v(t)]\le 0\quad\mbox{and}\quad\lim_{t\to\infty}\mathcal F[v(t)]=0\,.
\]
Moreover, $\frac d{dt}\mathcal F[v(t)]=0$ if and only if, for some $x_0\in\R$, $v_0(x)=v_\star(x-x_0)$ for any $x\in I$.\end{thm}

As a consequence of the above theorem, for every $v_0$ satisfying the assumptions stated in the above theorem, we have that
\begin{equation}\label{atlanta2}
\nrmlcpt{v_0'}2^2+\,\mathsf C\,\nrmlcpt {v_0}p^2-\,\frac4{(2-p)^2}\,\nrmlcpt {v_0} 2^2 \ge \lim_{t\to\infty}\mathcal F[v(t)] = 0 \ .
\end{equation}

This paper is organized as follows. Theorem \ref{Thm:duality} is proved in Section \ref{Sec:Duality} by implementing the strategy defined in (c). Optimality is checked in Section~\ref{Sec:Optimality}.

In Section \ref{Sec:monotonicity} we prove inequalities \eqref{GN1} and~\eqref{GN2} following the strategy defined in (b). The flow is easiest to construct after changes of variables which reduce the problem to critical interpolation inequalities involving the ultraspherical operator in case of \eqref{GN1} and a similar change of variables in case of \eqref{GN2}. More details are given in Section~\ref{Sec:monotonicity}, \emph{e.g.}, on the set of minimizers of the energy functionals, and some rigidity results are then stated in Section~\ref{Sec:Rigidity}.

In Section~\ref{Sec:Flows} we study some fast diffusion flows related to the difference of the left- and right-hand sides of inequalities \eqref{GN1} and~\eqref{GN2}, showing that sometimes these are gradient flows with respect to well-chosen distances introduced in \cite{pre05312043} and \cite{MR2448650}. For further developments also see~\cite{MR2565840}.

The Appendix~\ref{Sec:AppendixA} contains some auxiliary computations that are useful for flows and rigidity results, and common to \eqref{GN1} and~\eqref{GN2}. In Appendix~\ref{Sec:AppendixB}, for completeness we give a sketch of the method (a) applied to inequalities \eqref{GN1} and \eqref{GN2}.

\section{A duality approach using mass transportation methods}\label{Sec:Duality}

In this section we establish inequalities which relate the two sides of \eqref{Ineq:duality1} and \eqref{Ineq:duality2}. We also investigate the threshold case corresponding to $p=2$.

\subsection{Gagliardo-Nirenberg-Sobolev inequalities with $p>2$}

\begin{lem}\label{Lem:duality1} For any $p\in(2,\infty)$, we have
\[
\sup_{G\in\mathrm L_2^1(\R)\setminus\{0\}}\frac{\iRy{G^\frac{p+2}{3\,p-2}}}{\(\iRy{G\,|y|^2}\)^\frac{p-2}{3\,p-2}\,\(\iRy{G}\)^\frac4{3\,p-2}}\le\mathsf c_p\,\inf_{f\in\H^1(\R)\setminus\{0\}}\frac{\nrml {f'}2^\frac{2\,(p-2)}{3\,p-2}\,\nrml f2^\frac{2\,(p+2)}{3\,p-2}}{\nrml fp^\frac{4\,p}{3\,p-2}}\,.
\]
\end{lem}
\begin{proof} On the line, let $F$ and~$G$ be two probability densities and define the convex map $\varphi$ such that
\[
F(x)=G(\varphi'(x))\,\varphi''(x)\quad\forall\,x\in\R\,.
\]
Let us consider the change of variables $y=\varphi'(x)$, so that $dy=\varphi''(x)\,dx$ and compute, for some $\theta\in(0,1)$ to be fixed later, the integral
\[
\iRy{G^\theta}=\iR{G(\varphi'(x))^\theta\,\varphi''(x)}=\iR{F(x)^\theta\,(\varphi''(x))^{1-\theta}}\,.
\]
According to H\"older's inequality, for any $\alpha\in(0,\theta)$, 
\[
\iR{F(x)^\theta\,(\varphi''(x))^{1-\theta}}=\iR{F^{\theta-\alpha}\,F^\alpha\,(\varphi'')^{1-\theta}}\le\(\iR{F^{1-\frac\alpha\theta}}\)^\theta\,\(\iR{F^{\frac\alpha{1-\theta}}\,\varphi''}\)^{1-\theta}.
\]
Consider now the last integral and integrate by parts:
\[
\iR{F^{\frac\alpha{1-\theta}}\,\varphi''}=-\frac\alpha{1-\theta}\iR{F^{\frac\alpha{1-\theta}-1}\,F'\,\varphi'}=-\frac\alpha{1-\theta}\iR{F^{\frac\alpha{1-\theta}-\frac 1p}\,\varphi'\cdot F^{\frac 1p-1}\,F'}\,.
\]
If we choose $\alpha$ such that
\[
\frac\alpha{1-\theta}-\frac 1p=\frac 12\,,
\]
then we have
\[
\iR{F^{\frac\alpha{1-\theta}}\,\varphi''}=-\frac{\alpha\,p}{1-\theta}\iR{\sqrt F\,\varphi'\cdot \(F^{\frac 1p}\)'}\,.
\]
We deduce from the Cauchy-Schwarz inequality that
\[
\iR{F^{\frac\alpha{1-\theta}}\,\varphi''}\le\frac{\alpha\,p}{1-\theta}\(\iR{F\,|\varphi'|^2}\)^\frac 12\,\(\iR{\Big|\(F^\frac 1p\)'\Big|^2}\)^\frac 12.
\]
We also have
\[
\iR{F\,|\varphi'|^2}=\iRy{G\,|y|^2}\,.
\]
With $f:=F^\frac 1p$, we have found that
\[
\frac{\iRy{G^\theta}}{\(\iRy{G\,|y|^2}\)^\frac{1 - \theta}{2}} \le\(\frac{\alpha\,p}{1-\theta}\)^{1-\theta}\,\(\iR{F^{1-\frac\alpha\theta}}\)^\theta\,\(\iR{\Big|f'\Big|^2}\)^\frac{1-\theta}2.
\]
If we make the choices $p>2$ and
\[
1-\frac\alpha\theta=\frac 2p\,,
\]
then we have shown that
\[
\frac{\iRy{G^\theta}}{\(\iRy{G\,|y|^2}\)^\frac{1-\theta}2}\le\(\frac{\alpha\,p}{1-\theta}\)^{1-\theta}\,\(\iR{f^2}\)^\theta\,\(\iR{\Big|f'\Big|^2}\)^\frac{1-\theta}2.
\]
with
\[
\theta=\frac{p+2}{3\,p-2}\quad\mbox{and}\quad \alpha=\frac{(p-2)\,(p+2)}{p\,(3\,p-2)}\,.
\]
Taking into account the homogeneity, this establishes \eqref{Ineq:duality1}, where the infimum is now taken over all non-trivial functions $f$ in $\mathrm H^1(\R)$ and the supremum is taken over all non-trivial non-negative integrable functions with finite second moment. The computation is valid for any $p\in(2,\infty)$.
\end{proof}

The generalization of our method to higher dimensions $d\ge2$ would involve the replacement of $\iR{F^{\frac\alpha{d\,(1-\theta)}}\,\varphi''}$ by $\int_{\R^d}F^{\frac\alpha{d\,(1-\theta)}}\,\(\mathrm{det}\,\mathrm{Hess}(\varphi)\)^{1/d}\,dx$ in order to use the fact that $\(\mathrm{det}\,\mathrm{Hess}(\varphi)\)^{1/d}\le\frac 1d\,\Delta\varphi$ by the arithmetic-geometric inequality. The reader is invited to check that the system
\[
\frac{\theta-\alpha}{1-d\,(1-\theta)}=\frac2p\,,\quad\frac\alpha{d\,(1-\theta)}=\frac1p+\frac12
\]
has no solutions $(\alpha,\theta)$ such that $\theta\in(0,1)$ if $d\ge2$.

\subsection{Gagliardo-Nirenberg-Sobolev inequalities with $1<p<2$}

\begin{lem}\label{Lem:duality2} For any $p\in(1,2)$, we have
\[
\sup_{G\in\mathrm L_2^1(\R)\setminus\{0\}}\frac{\iRy{G^\frac2{4-p}}}{\(\iRy{G\,|y|^2}\)^\frac{2-p}{2\,(4-p)}\,\(\iRy{G}\)^\frac{p+2}{2\,(4-p)}}\le\mathsf c_p\,\inf_{f\in\H^1(\R)\setminus\{0\}}\,\frac{\nrml{f'}2^\frac{2-p}{4-p}\,\nrml fp^\frac{2\,p}{4-p}}{\nrml f2^\frac{p+2}{4-p}}\,.
\]
\end{lem}
\begin{proof} We start as above by writing
\begin{multline*}
\iRy{G^\theta}\le\(\iR{F^{1-\frac\alpha\theta}}\)^\theta\,\(\iR{F^{\frac\alpha{1-\theta}}\,\varphi''}\)^{1-\theta}\\
=\(\iR{F^{1-\frac\alpha\theta}}\)^\theta\,\(-\frac\alpha{1-\theta}\iR{F^{\frac\alpha{1-\theta}-\frac 12}\,\varphi'\cdot F^{\frac 12-1}\,F'}\)^{1-\theta}\\
=\(\iR{F^{1-\frac\alpha\theta}}\)^\theta\,\(-\frac{2\,\alpha}{1-\theta}\iR{F^{\frac\alpha{1-\theta}-\frac 12}\,\varphi'\cdot \(\sqrt F\)'}\)^{1-\theta}.
\end{multline*}
and choose $\alpha$ and $\theta$ such that
\[
1-\frac\alpha\theta=\frac p2\quad\mbox{and}\quad\frac\alpha{1-\theta}-\frac 12=\frac 12\;,
\]
for some $p\in(1,2)$. With $f=\sqrt F$ and
\[
\theta=\frac2{4-p}=1-\alpha\,,
\]
the r.h.s.~can be estimated as
\begin{multline*}
\(\iR{f^p}\)^\theta\,\(-\frac{2\,\alpha}{1-\theta}\iR{\sqrt F\,\varphi'\cdot f'}\)^{1-\theta}\\
\le2^\frac{2-p}{4-p}\(\iR{f^p}\)^\frac2{4-p}\(\iR{|f'|^2}\)^\frac{2-p}{2\,(4-p)}\(\iR{F\,|\varphi'|^2}\)^\frac{2-p}{2\,(4-p)}\,,
\end{multline*}
using a Cauchy-Schwarz inequality. We also have $\iR{F\,|\varphi'|^2}=\iRy{G\,|y|^2}$ as in the proof of Lemma~\ref{Lem:duality1}. Taking into account the homogeneity, this establishes \eqref{Ineq:duality2} where the infimum is now taken over all non-trivial functions $f$ in $\mathrm L^p(\R)$ whose derivatives are square integrable and the supremum is taken over all non-trivial non-negative integrable functions with finite second moment. The computation is valid for any $p\in(1,2)$.\end{proof}

The generalization of our method to higher dimensions $d\ge2$ would involve the replacement of $\iR{F^{\frac\alpha{d\,(1-\theta)}}\,\varphi''}$ by $\int_{\R^d}F^{\frac\alpha{d\,(1-\theta)}}\,\(\mathrm{det}\,\mathrm{Hess}(\varphi)\)^{1/d}\,dx$ in order to use the fact that $\(\mathrm{det}\,\mathrm{Hess}(\varphi)\)^{1/d}\le\frac 1d\,\Delta\varphi$ by the arithmetic-geometric inequality. The reader is invited to check that the system
\[
\frac{\theta-\alpha}{1-d\,(1-\theta)}=\frac p2\,,\quad\frac\alpha{d\,(1-\theta)}=1
\]
has no solutions $(\alpha,\theta)$ such that $\theta\in(0,1)$ if $d\ge2$.

\subsection{The threshold case: logarithmic Sobolev inequality}

We can consider the limit $p\to2$ in \eqref{Ineq:duality1}. If we take the logarithm of both sides of the inequality, multiply by $\frac4{p-2}$ and pass to the limit as $p\to2_+$, we find that
\begin{multline*}
-\,2\,\frac{\iRy{G\,\log G}}{\iRy G}-\log\iRy{|y|^2\,G}+3\log\iRy G-1\\
\le\log\iR{|f'|^2}-\log\iR{|f|^2}-2\,\frac{\iR{|f|^2\,\log |f|^2}}{\iR{|f|^2}}+\log\(\frac 4e\)\,.
\end{multline*}
Hence we recover the following well-known fact.
\begin{lem}\label{Lem:duality3}
\begin{multline*}
\sup_{G\in\mathrm L_2^1(\R)\setminus\{0\}}\left[\log\(\frac{\nrml G1^3}{2\,\pi\iRy{|y|^2\,G}}\)-\,2\,\frac{\iRy{G\,\log G}}{\nrml G1}-1\right]\\
\le\inf_{f\in\H^1(\R)\setminus\{0\}}\,\left[\log\(\frac2{\pi\,e}\,\frac{\nrml{f'}2^2}{\nrml f2^2}\)-2\,\frac{\iR{|f|^2\,\log |f|^2}}{\nrml f2^2}\right]\,.
\end{multline*}
\end{lem}
A similar computation can be done based on \eqref{Ineq:duality2}. For a direct approach based on mass transportation, in any dimension, we may refer to the result established by D.~Cordero-Erausquin in \cite{Cordero00}.

\section{Optimality and best constants}\label{Sec:Optimality}

A careful investigation of the equality cases in all inequalities used in the computations of Section~\ref{Sec:Duality} shows that inequalities in Lemmata~\ref{Lem:duality1}, \ref{Lem:duality2} and \ref{Lem:duality3} can be made equalities for optimal functions as was done, for instance, in \cite{MR2032031}. In this section, we directly investigate the cases of optimality in~\eqref{Ineq:duality1} and~\eqref{Ineq:duality2}, prove the equalities in Theorem~\ref{Thm:duality} and compute the values of the corresponding constants $\mathsf C_{\rm GN}(p)$.

\subsection{Duality in the Gagliardo-Nirenberg-Sobolev inequalities: case $p>2$.}\label{Sec:GN1}
Let us compute
\be{Eqn:C1}
\mathsf C_1(p):=\mathsf c_p\,\inf_{f\in\H^1(\R)\setminus\{0\}}\frac{\nrml {f'}2^\frac{2\,(p-2)}{3\,p-2}\,\nrml f2^\frac{2\,(p+2)}{3\,p-2}}{\nrml fp^\frac{4\,p}{3\,p-2}}\,.
\ee
The infimum is achieved by
\[
f_\star(x)=\frac 1{(\cosh x)^\frac2{p-2}}\quad\forall\,x\in\R\,,
\]
which solves the equation
\[
-\,(p-2)^2\,f''+4\,f-\,2\,p\,f^{p-1}=0\,.
\]
See Appendix~\ref{Sec:AppendixB} for more details. With the formulae
\begin{multline*}
\mathsf I_2:=\iR{f_\star^2}=\frac{\sqrt\pi\,\Gamma\(\frac2{p-2}\)}{\Gamma\(\frac{p+2}{2\,(p-2)}\)}\,,\quad\iR{f_\star^p}=\frac4{p+2}\iR{f_\star^2}\\
\mbox{and}\quad\iR{|f_\star'|^2}=\frac4{(p-2)\,(p+2)}\iR{f_\star^2}\,,
\end{multline*}
one can check that the r.h.s.~in~\eqref{Ineq:duality1} can be computed and amounts to
\[
\mathsf C_1(p)=\frac{(p+2)^\frac{p+2}{3\,p-2}}{4^\frac4{3\,p-2}\,(p-2)^\frac{p-2}{3\,p-2}}\,\mathsf I_2^\frac{2\,(p-2)}{3\,p-2}\,.
\]
On the other hand, the supremum in \eqref{Ineq:duality1} is achieved by
\[
G_\star(y)=\frac1{(1+y^2)^q}\quad\forall\,y\in\R\,,
\]
with
\[
q=\frac{3\,p-2}{2\,(p-2)}\,.
\]
Using the function
\[
h(q):=\int_\R\frac{dy}{(1+y^2)^q}=\frac{\sqrt\pi\,\Gamma\(q-\frac12\)}{\Gamma(q)}\,,
\]
it is easy to observe that
\[
\iRy{G_\star}=h(q)\,,\quad\iRy{G_\star\,|y|^2}=\frac{h(q)}{2\,q-3}\quad\mbox{and}\quad\iRy{G_\star^\frac{p+2}{3\,p-2}}=\frac{2\,(q-1)}{2\,q-3}\,h(q)\,,
\]
and recover that the l.h.s.~in~\eqref{Ineq:duality1} also amounts to $\mathsf C_1(p)$. With Lemma~\ref{Lem:duality1}, this completes the proof of Theorem~\ref{Thm:duality} when $p>2$. This also shows that the best constant in \eqref{GN1} is 
\[
\mathsf C_{\rm GN}(p)=\(\frac{\mathsf C_1(p)}{\mathsf c(p)}\)^\frac{3\,p-2}{4\,p}.
\]
with $\mathsf C_1(p)$ and $\mathsf c(p)$ given by \eqref{Eqn:C1} and \eqref{Eqn:cp} respectively.

\subsection{Second case in the Gagliardo-Nirenberg-Sobolev inequalities: case $1<p<2$.}\label{Sec:GN2}

Let us compute
\be{Eqn:C2}
\mathsf C_2(p):=\mathsf c_p\,\inf_{f\in\H^1(\R)\setminus\{0\}}\,\frac{\nrml{f'}2^\frac{2-p}{4-p}\,\nrml fp^\frac{2\,p}{4-p}}{\nrml f2^\frac{p+2}{4-p}}\,.
\ee
The infimum is achieved by
\[
f_*(x)=(\cos x)^\frac2{2-p}\quad\forall\,x\in\left[-\tfrac\pi2,\tfrac\pi2\right]\,,\quad f_*(x)=0\quad\forall\,x\in\R\setminus\left[-\tfrac\pi2,\tfrac\pi2\right]\,,
\]
which solves the equation
\[
-\,(2-p)^2\,f''-4\,f+\,2\,p\,f^{p-1}=0\,.
\]
With the formulae
\begin{multline*}
\mathsf J_2:=\iR{f_*^2}=\frac{\sqrt\pi\,\Gamma\(\frac{6-p}{2\,(2-p)}\)}{\Gamma\(\frac{4-p}{2-p}\)}\,,\quad\iR{f_*^p}=\frac4{p+2}\iR{f_*^2}\\
\mbox{and}\quad\iR{|f_*'|^2}=\frac4{(2-p)\,(2+p)}\iR{f_*^2}\,,
\end{multline*}
one can check that the r.h.s.~in~\eqref{Ineq:duality2} can be computed and amounts to
\[
\mathsf C_2(p)=4\,(2+p)^{-\frac{6-p}{2\,(4-p)}}\,(2-p)^{-\frac{2-p}{2\,(4-p)}}\,\mathsf J_2 ^\frac{2-p}{4-p}\,.
\]

On the other hand, the supremum in \eqref{Ineq:duality2} is achieved by
\[
G_*(y)=\frac1{(1+y^2)^q}\quad\forall\,y\in\R\,,
\]
with
\[
q=\frac{4-p}{2-p}\,.
\]
Using the function $h(q)$ as in the first case, $h(\frac{4-p}{2-p})=\mathsf J_2$ and the relations
\[
\iRy{G_*}=h(q)\,,\quad\iRy{G_*\,|y|^2}=\frac{h(q)}{2\,q-3}\quad\mbox{and}\quad\iRy{G_*^\frac{p+2}{3\,p-2}}=\frac{2\,(q-1)}{2\,q-3}\,h(q)\,,
\]
we recover that the l.h.s.~in~\eqref{Ineq:duality1} also amounts to $\mathsf C_2(p)$. With Lemma~\ref{Lem:duality2}, this completes the proof of Theorem~\ref{Thm:duality} when $p<2$. This also shows that the best constant in \eqref{GN2} is 
\[
\mathsf C_{\rm GN}(p)=\(\frac{\mathsf C_2(p)}{\mathsf c(p)}\)^\frac{4-p}{2+p}.
\]
with $\mathsf C_2(p)$ and $\mathsf c(p)$ given by \eqref{Eqn:C2} and \eqref{Eqn:cp} respectively.

\subsection{Consistency with the logarithmic Sobolev inequality.} The reader is invited to check that for $p=2$, we have $\lim_{p\to2_+}\mathsf C_1(p)=\lim_{p\to2_-}\mathsf C_2(p)=1$ and
\[
4\,\lim_{p\to2_+}\frac{\mathsf C_1(p)-1}{p-2}=1+\log\(2\,\pi\)=4\,\lim_{p\to2_-}\frac{1-\mathsf C_2(p)}{2-p}\,.
\]
This is consistent with the fact that we have equality in Lemma~\ref{Lem:duality3} and can actually write
\begin{cor}\label{Cor:duality3}
\begin{multline*}
\sup_{G\in\mathrm L_2^1(\R)\setminus\{0\}}\left[\log\(\frac{\nrml G1^3}{2\,\pi\iRy{|y|^2\,G}}\)-\,2\,\frac{\iRy{G\,\log G}}{\nrml G1}-1\right]\\
=\inf_{f\in\H^1(\R)\setminus\{0\}}\,\left[\log\(\frac2{\pi\,e}\,\frac{\nrml{f'}2^2}{\nrml f2^2}\)-2\,\frac{\iR{|f|^2\,\log |f|^2}}{\nrml f2^2}\right]=0\,.
\end{multline*}
\end{cor}
The reader is invited to check that equality is realized by
\[
G(x)=|f(x)|^2=\frac{e^{-\frac{|xl^2}2}}{\sqrt{2\,\pi}}\,,\quad x\in\R.
\]
Hence we recover not only the logarithmic Sobolev inequality in Weissler's form \cite{MR479373}, but also the fact that the equality case is achieved by Gaussian functions.

\section{Gagliardo-Nirenberg-Sobolev inequalities, monotonicity and flows}\label{Sec:monotonicity}

This section is devoted to the proof of Theorems~\ref{thm:flow1} and~\ref{thm:flow2}, and their consequences.

\subsection{Inequality \eqref{atlanta1} (case $p>2$) and the ultraspherical operator}\label{flowplargerthan2}

In this section we reduce the inequality \eqref{atlanta1} on the line to a weighted problem on the interval $(-1,1)$. For $p>2$, Gagliardo-Nirenberg-Sobolev inequalities on the line indeed are equivalent to critical interpolation inequalities for the ultraspherical operator (see \cite{MR1231419}; these inequalities correspond to the well-known inequalities on the sphere \cite{MR1134481,MR1230930,DEKL} when the dimension is an integer).

In order to make our strategy easier to understand, the proofs have been divided in a series of statements. Some of them go beyond what is required for the proofs of the results in Section~\ref{Sec:Intro}.

\newclaim{The inequality \eqref{atlanta1} on the line is equivalent to the critical problem for the ultraspherical operator.}

Recall that inequality \eqref{atlanta1} is given by
\be{atlanta3}
\iR{|v'|^2}+\frac4{(p-2)^2}\iR{|v|^2}\ge\mathsf C\(\iR{|v|^p}\)^\frac2p\,.
\ee
With
\[
z(x)=\tanh x\,,\quad v_\star=(1-z^2)^\frac 1{p-2}\quad\mbox{and}\quad v(x)=v_\star(x)\,f(z(x))\,,
\]
so that, as seen in Section \ref{Sec:GN1}, equality is achieved for $f=1$, \emph{i.e.} with
\[
\mathsf C=\frac{2\,p}{(p-2)^2}\(\iR{|v_\star|^p}\)^{1-\frac2p}\,,
\]
and, if we let $\nu(z):=1-z^2$, the above inequality is equivalent~to
\be{Ineq:Ultraspherical}
\izp{|f'|^2\,\nu}+\frac{2\,p}{(p-2)^2}\izp{|f|^2}\ge\frac{2\,p}{(p-2)^2}\(\izp{|f|^p}\)^\frac2p
\ee
where $d\nu_p$ denotes the probability measure $d\nu_p(z):=\frac1{\zeta_p}\,\nu^\frac2{p-2}\,dz$, $\zeta_p:=\sqrt\pi\,\frac{\Gamma\(\frac p{p-2}\)}{\Gamma\(\frac{3\,p-2}{2\,(p-2)}\)}$. Integration
by parts leads to
\[
\izp{|f'|^2\,\nu}=-\izp{f\,(\L f)}\quad\mbox{where}\quad \L f:=\nu\,f''-\,\frac{2\,p}{p-2}\,z\,f' \ .
\]
If we set
\[
d=\frac{2\,p}{p-2}\quad\Longleftrightarrow\quad p=\frac{2\,d}{d-2}\,.
\]
the operator $\L$ is the ultraspherical operator. Thus, we see that the inequality \eqref{atlanta3} on the line is equivalent to a problem that involves the $d$-ultraspherical operator.

When $d$ is an integer, it is known that the inequality for the ultraspherical operator \eqref{Ineq:Ultraspherical} is equivalent to an inequality on the $d$-dimensional sphere (see for instance \cite{DEKL} and references therein). We are now interested in the monotonicity of the functional
\[
f\mapsto\mathsf F[f]:=\izp{|f'|^2\,\nu}+\frac{2\,p}{(p-2)^2}\izp{|f|^2}-\frac{2\,p}{(p-2)^2}\(\izp{|f|^p}\)^\frac2p
\]
along a well-chosen nonlinear flow.

\newclaim{There exists a nonlinear flow along which our functional is monotone non-increasing.}

With the above notations, the problem is reduced to the computation on the $d$-dimensional sphere in the ultraspherical setting. Here we adapt the strategy of \cite{DEKL} and \cite{dolbeault:hal-00784887}. We recall (see Proposition~\ref{Prop:2id} in Appendix~\ref{Sec:AppendixA}, with $\aa=1$ and $\bb=\frac d2-1$) that
\[
\izp{(\L u)^2}=\izp{|u''|^2\;\nu^2}+d\izp{|u'|^2\;\nu}
\]
and
\[
\izp{(\L u)\,\frac{|u'|^2}u\,\nu}=\frac d{d+2}\izp{\frac{|u'|^4}{u^2}\;\nu^2}-\,2\,\frac{d-1}{d+2}\izp{\frac{|u'|^2\,u''}u\;\nu^2}\,.
\]
On $(-1,1)$, let us consider the flow
\[
u_t=u^{2-2\beta}\(\L u+\kappa\,\frac{|u'|^2}u\,\nu\)
\]
and notice that
\[
\frac d{dt}\izp{u^{\beta p}}=\beta\,p\,(\kappa-\beta\,(p-2)-1)\izp{u^{\beta(p-2)}\,|u'|^2\,\nu}\,,
\]
so that $\overline u=\(\izp{u^{\beta p}}\)^{1/(\beta p)}$ is preserved if $\kappa=\beta\,(p-2)+1$. With $\beta=\frac 4{6-p}$, a lengthy computation shows that
\begin{multline} \label{atlanta4}
\frac 1{2\,\beta^2}\,\frac d{dt}\izp{\(|(u^\beta)'|^2\,\nu+\frac d{p-2}\,\(u^{2\beta}-\overline u^{2\beta}\)\)}\\
=-\izp{\(\L u+(\beta-1)\,\frac{|u'|^2}u\,\nu\)\(\L u+\kappa\,\frac{|u'|^2}u\,\nu\)}+\frac d{p-2}\,\frac{\kappa-1}\beta\izp{|u'|^2\,\nu}\\
=-\izp{|u''|^2\,\nu^2}+\,2\,\frac{d-1}{d+2}\,(\kappa+\beta-1)\izp{u''\,\frac{|u'|^2}u\,\nu^2}\\
-\left[\kappa\,(\beta-1)+\,\frac d{d+2}\,(\kappa+\beta-1)\right]\izp{\frac{|u'|^4}{u^2}\,\nu^2}\\
=-\izp{\left|u''-\frac{p+2}{6-p}\,\frac{|u'|^2}u\right|^2\nu^2}\,. 
\end{multline}

The choice of the change of variables $f=u^\beta$ was motivated by the fact that the last term in the above identities is two-homogeneous in $u$, thus making the completion of the square rather simple. It is also a natural extension of the case that can be carried out with a linear flow (see \cite{DEKL}, and \cite{MR808640} for a much earlier result in this direction). In the above computations $p=6$ seems to be out of reach, but as we see below, this case can also be treated by writing the flow in the original variables.

\newclaim{There is no restriction on the range of the exponents.}

With $f=u^\beta$, the problem can be rewritten in the setting of the ultraspherical operator using the flow
\[
f_t=f^{1-\frac p2}\left[\L f+\tfrac p2\,(1-\,z^2)\,\frac{|f'|^2}f\right]\,,
\]
and we notice that there is no more singularity when $p=6$ since
\begin{multline*}
\frac d{dt}\left[\izp{|f'|^2\,\nu}+\frac{2\,p}{(p-2)^2}\izp{|f|^2}-\mathsf C\(\izp{|f|^p}\)^\frac2p\right]\\
=-\,2\izp{f^{1-\frac p2}\,\left|f''-\frac p2\,\frac{|f'|^2}f\right|^2\nu^2}\,.
\end{multline*}
We get the flow on the line by undoing the change of variables: the function $v(t,x)=v_\star(x)\,f(t,z(x))$ solves
\[
v_t=\frac{v^{1-\frac p2}}{\sqrt{1-z^2}}\left[v''+\,\frac{2\,p}{p-2}\,z\,v'+\,\frac p2\,\frac{|v'|^2}v+\,\frac2{p-2}\,v\,\right]\,,
\]
and we find that
\begin{multline*}
\frac d{dt}\left[\iR{|v'|^2}+\frac4{(p-2)^2}\iR{|v|^2}-\mathsf C\(\iR{|v|^p}\)^\frac2p\right]\\
=-\,2\iR{\frac1{(1-z^2)^2}\(\frac v{v_\star}\)^{1-\frac p2}\left|v''-\frac p2\,\frac{|v'|^2}v+\frac2{p-2}\,v\right|^2}\,.
\end{multline*}

\newclaim{There exists a one-dimensional family of minimizers for $\mathsf F$.}

A function $f$ is in the constant energy manifold, \emph{i.e.}, $\mathcal F[f(t)]$ does not depend on $t$, if and only if $(f'\,f^{-p/2})'=0$, that is, $f(z)=(\mathsf a+\mathsf b\,z)^{-\frac2{p-2}}$. However, none of the elements of that manifold, except the one corresponding to $\aa=1$ and $\bb=0$, are left invariant under the action of the flow and the coefficients $\mathsf a$ and $\mathsf b$ obey to the system of ordinary differential equations
\[
\frac{d\,\mathsf a}{dt}=-\frac{2\,p}{p-2}\,\mathsf b^2\quad\mbox{and}\quad\frac{d\,\mathsf b}{dt}=-\frac{2\,p}{p-2}\,\mathsf a\,\mathsf b\,.
\]
The reader is invited to check that on the line, such functions are given by
\[
v(t,x)=\frac1{\cosh(x+x(t))^\frac2{p-2}}\quad\forall\,(t,x)\in\R^2\,,
\]
with $\mathsf a(t)=\cosh(x(t))$ and $\mathsf b(t)=\sinh(x(t))$. A straightforward but painful computation provides an explicit expression for $t\mapsto x(t)$. 

\newclaim{The inequality on the line can be reinterpreted using the stereographic projection and the Emden-Fowler transformation.}

Inequality \eqref{Ineq:Ultraspherical} (in ultraspherical coordinates) is
\[
\izp{|f'|^2\,\nu}+\frac{2\,p}{(p-2)^2}\izp{|f|^2}\ge\frac{2\,p}{(p-2)^2}\(\izp{|f|^p}\)^\frac2p\,,
\]
where $d\nu_p$ denotes the probability measure $d\nu_p(z):=\frac1{\zeta_p}\,\nu^\frac2{p-2}\,dz$, $\zeta_p:=\sqrt\pi\,\frac{\Gamma\(\frac p{p-2}\)}{\Gamma\(\frac{3\,p-2}{2\,(p-2)}\)}$ and
\[
d=\frac{2\,p}{p-2}\quad\Longleftrightarrow\quad p=\frac{2\,d}{d-2}\,.
\]
Since $\frac{2\,p}{(p-2)^2}=\frac 14\,d\,(d-2)$, the above inequality can be rewritten as
\[
\frac 4{d\,(d-2)}\izp{|f'|^2\,\nu}+\izp{|f|^2}\ge\(\izp{|f|^p}\)^\frac2p\,.
\]
Assume that
\[
f(z)=(1-z)^{1-\frac d2}\,u(r)\quad\mbox{with}\quad z=1-\frac 2{1+r^2}\quad\Longleftrightarrow\quad r=\sqrt{\frac{1+z}{1-z}}\,.
\]
When $d$ is an integer, this first change of variables corresponds precisely to the \emph{stereographic projection}. Then by direct computation we find that
\[
\frac 4{d\,(d-2)}\izp{|f'|^2\,\nu}+\izp{|f|^2}=\frac 4{d\,(d-2)}\,\frac1{\zeta_p}\ir{|u'|^2}\,,
\]
\[
\izp{|f|^p}=\frac1{\zeta_p}\ir{|u|^p}\,,
\]
so that the inequality becomes
\[
\ir{|u'|^2}\ge\frac14\,d\,(d-2)\,\zeta_p^{1-\frac 2p}\(\ir{|u|^p}\)^\frac 2p\,.
\]
Let $u(r):=r^{1-\frac d2}\,v(x)$ with $x=\log r$. This second change of variables is the \emph{Emden-Fowler transformation}. Then we get
\[
\iR{|v'|^2}+\frac 14\,(d-2)^2\iR{|v|^2}\ge \frac14\,d\,(d-2)\,\zeta_p^{1-\frac 2p}\(\iR{|v|^p}\)^\frac 2p\,.
\]
Recalling how $p$ and $d$ are related, this means
\[
\iR{|v'|^2}+\frac4{(p-2)^2}\,\iR{|v|^2}\ge\frac{2\,p}{(p-2)^2}\,\zeta_p^{1-\frac 2p}\(\iR{|v|^p}\)^\frac 2p\,.
\]
Collecting the two changes of variables, what has been done amounts to the change of variables
\[
z(x)=\tanh x\,,\quad v_\star=\nu^\frac 1{p-2}\quad\mbox{and}\quad v(x)=v_\star(x)\,f(z(x))\,.
\]
This explains why \emph{the problem on the line is equivalent to the critical problem on the sphere} (when $d$ is an integer) or why the problem on the line is equivalent to the critical problem for the ultraspherical operator.

\subsection{Inequality \eqref{atlanta2} (Case $1<p<2$)}\label{flowplessthan2}

The computations for $p>2$ and $p<2$ are similar. This is what occurs in the construction of a nonlinear flow. For the convenience of the reader, we also subdivide this section in a series of claims.

\newclaim{The interpolation inequality is equivalent to a weighted interpolation inequality on the line.}

The Gagliardo-Nirenberg-Sobolev inequality on the line with $p\in(1,2)$ is equivalent~to
\begin{multline}\label{Ineq:duality2bis}
\iR{|v'|^2}-\frac4{(2-p)^2}\iR{|v|^2}\ge-\,\frac{2\,p}{(2-p)^2}\iR{|v_*|^p}=-\,\mathsf C\(\iR{|v|^p}\)^\frac 2p\\
\forall\,v\in\H^1(\R)\cap\mathrm L^p(\R)\quad\mbox{such that}\quad\iR{|v|^p}=\iR{|v_*|^p}\,.
\end{multline}
Following the computations of Section \ref{Sec:GN2}, we have
\[
\mathsf C=\frac{2\,p}{(2-p)^2}\(\iR{|v_*|^p}\)^{1-\frac 2p}\,.
\]
Assume that $v$ is supported in the interval $(-\frac\pi2,\frac\pi2)$. With $\xi(\z)=1+\z^2$, so that for any $x\in(-\frac\pi2,\frac\pi2)$
\[
\z(x)=\tan x\,,\quad v_*=\xi(\z)^{-\frac 1{2-p}}\quad\mbox{and}\quad v(x)=v_*(x)\,f(\z(x))\,,
\]
the inequality is equivalent to
\be{Ineq:zp<2}
\izpp{|f'|^2\,\xi}+\frac{2\,p}{(2-p)^2}\(\izpp{|f|^p}\)^\frac2p\ge\frac{2\,p}{(2-p)^2}\izpp{|f|^2}\,,
\ee
where $d\xi_p$ denotes the probability measure $d\xi_p(\z):=\frac1{\zeta_p}\,\xi^{-\frac2{2-p}}\,d\z$ with $\zeta_p:=\sqrt\pi\,\frac{\Gamma\(\frac{2+p}{2\,(2-p)}\)}{\Gamma\(\frac2{2-p}\)}$. Let us define
\[
\L f:=\xi\,f''-\,\frac{2\,p}{2-p}\,\z\,f'\,.
\]
Notice that $\mathsf C=\frac{2\,p}{(2-p)^2}\,\zeta_p^{1-\frac 2p}$. Inequality~\eqref{Ineq:zp<2} is equivalent to the inequality~\eqref{Ineq:duality2bis} and is therefore optimal. We are interested in the monotonicity of the functional
\[
f\mapsto\mathsf F[f]:=\izpp{|f'|^2\,\xi}+\frac{2\,p}{(2-p)^2}\(\izpp{|f|^p}\)^\frac2p-\frac{2\,p}{(2-p)^2}\izpp{|f|^2}
\]
along a well-chosen nonlinear flow. We will first establish two identities.

\newclaim{There are also two key identities in the case $p\in(1,2)$.}

As a preliminary observation, we can observe that
\[
\left[\frac d{dx},\L\right]u=2\,\z\,u''-\,\frac{2\,p}{2-p}\,u'\,,
\]
so that we immediately get
\begin{multline*}
\izpp{(\L u)^2}=-\izpp{\xi\,u'\,(\L u)'}\\
=-\izpp{\xi\,u'\,(\L u')}-\izpp{\xi\,u'\,(2\,\z\,u''-\,\frac{2\,p}{2-p}\,u')}\\
\hspace*{3cm}=\izpp{\xi\,(\xi\,u')'\,u''}-\izpp{\xi\,u'\,(2\,\z\,u''-\,\frac{2\,p}{2-p}\,u')}\\
=\izpp{|u''|^2\;\xi^2}+\frac{2\,p}{2-p}\izpp{|u'|^2\;\xi}
\end{multline*}
and
\begin{multline*}
\izpp{(\L u)\,\frac{|u'|^2}u\,\xi}=\izpp{\xi\,u'\,\(\frac{|u'|^2\,u'}{u^2}\,\xi-2\,\frac{u'\,u''}u\,\xi-2\,\z\,\frac{|u'|^2}u\)}\\
=\frac p{2\,(p-1)}\izpp{\frac{|u'|^4}{u^2}\;\xi^2}-\frac{p+2}{2\,(p-1)}\izpp{\frac{|u'|^2\,u''}u\,\xi^2}\,,
\end{multline*}
since
\begin{multline*}
\izpp{\frac{|u'|^2\,u''}u\,\xi^2}=\frac 13\int_\R{(|u'|^2\,u')'\,\frac 1u\,\xi^{-2\,\frac{p-1}{2-p}}\;d\z}\\
=\frac 13\izpp{\frac{|u'|^4}{u^2}\;\xi^2}+\frac 43\,\frac{p-1}{2-p}\izpp{\frac{|u'|^2\,u'}u\;\z\,\xi}\,,
\end{multline*}
and hence
\[
\izpp{\frac{|u'|^2\,u'}u\;\z\,\xi}=\frac{3\,(2-p)}{4\,(p-1)}\izpp{\frac{|u'|^2\,u''}u\,\xi^2}-\frac{2-p}{4\,(p-1)}\izpp{\frac{|u'|^4}{u^2}\;\xi^2}\,.
\]

Notice that these two identities enter in the general framework which is described in Appendix~\ref{Sec:AppendixA} with $\xi(\z)=1+\z^2$, $\aa=1$ and $\bb=-\frac2{2-p}$. Since they are not as standard as the ones corresponding to the ultraspherical operator, we have given a specific proof.

\newclaim{There exists a nonlinear flow along which our functional is monotone non-increasing.}

On $\R$, let us consider the flow
\[
u_t=u^{2-2\beta}\(\L u+\kappa\,\frac{|u'|^2}u\,\xi\)\,,
\]
and notice that
\[
\frac d{dt}\izpp{u^{\beta p}}=\beta\,p\,(\kappa-\beta\,(p-2)-1)\izpp{u^{\beta(p-2)}\,|u'|^2\,\xi}\,,
\]
so that $\overline u=\(\izpp{u^{\beta p}}\)^{1/(\beta p)}$ is preserved if $\kappa=\beta\,(p-2)+1$. Using the above estimates, a straightforward computation shows that
\begin{multline*}
\frac 1{2\,\beta^2}\,\frac d{dt}\izpp{\(|(u^\beta)'|^2\,\xi-\frac{2\,p}{(2-p)^2}\,\(u^{2\beta}-\overline u^{2\beta}\)\)}\\
=-\izpp{\(\L u+(\beta-1)\,\frac{|u'|^2}u\,\xi\)\(\L u+\kappa\,\frac{|u'|^2}u\,\xi\)}-\frac{2\,p}{(2-p)^2}\,\frac{\kappa-1}\beta\izpp{|u'|^2\,\xi}\\
=-\izpp{|u''|^2\,\xi^2}+\frac{p+2}{2\,(p-1)}\,(\kappa+\beta-1)\izpp{u''\,\frac{|u'|^2}u\,\xi^2}\\
-\left[\kappa\,(\beta-1)+\frac p{2\,(p-1)}\,(\kappa+\beta-1)\right]\izpp{\frac{|u'|^4}{u^2}\,\xi^2}\,.
\end{multline*}
With
\[
\beta=\frac4{6-p}\,,
\]
we get
\[
\frac d{dt}\izpp{\(|(u^\beta)'|^2\,\xi-\frac{2\,p}{(2-p)^2}\,\(u^{2\beta}-\overline u^{2\beta}\)\)}=\,-\,2\,\beta^2\izpp{\left|u''-\frac{p+2}{6-p}\,\frac{|u'|^2}u\right|^2\,\xi^2}\,.
\]

\newclaim{The flow can be rewritten in original variables.}

With $f=u^\beta$, the problem can be rewritten using the flow
\[
f_t=f^{1-\frac p2}\left[\L f+\tfrac p2\,\xi\,\frac{|f'|^2}f\right]\,,
\]
and we find that
\begin{multline*}
\frac d{dt}\left[\izpp{|f'|^2\,\xi}-\frac{2\,p}{(2-p)^2}\izpp{|f|^2}+\mathsf C\(\izpp{|f|^p}\)^\frac2p\right]\\
=-\,2\izpp{f^{1-\frac p2}\,\left|f''-\frac p2\,\frac{|f'|^2}f\right|^2\xi^2}\,.
\end{multline*}
We get the flow on $(-\frac\pi2,\frac\pi2)$ by undoing the change of variables: the function $v(t,x)=v_*(x)\,f(t,\z(x))$ solves
\[
v_t=\frac{v^{1-\frac p2}}{\sqrt{1+\z^2}}\left[v''+\,\frac{2\,p}{2-p}\,\z\,v'+\,\frac p2\,\frac{|v'|^2}v+\,\frac2{2-p}\,v\,\right]\,,
\]
and we find that
\begin{multline*}
\frac d{dt}\left[\ilcpt{|v'|^2}-\frac4{(2-p)^2}\ilcpt{|v|^2}+\mathsf C\(\ilcpt{|v|^p}\)^\frac2p\right]\\
=-\,2\iR{\frac1{(1+\z^2)^2}\(\frac v{v_*}\)^{1-\frac p2}\left|v''-\frac p2\,\frac{|v'|^2}v+\frac2{2-p}\,v\right|^2}\,.
\end{multline*}

\subsection{Consequences: monotonicity of the functionals associated to the Gagliardo-Nirenberg-Sobolev inequalities along the flows}

With the results of Sections~\ref{flowplargerthan2} and~\ref{flowplessthan2}, the proofs of Theorems~\ref{thm:flow1} and~\ref{thm:flow2} are rather straightforward and left to the reader.

\section{Rigidity results}\label{Sec:Rigidity}

In the case of compact manifolds with positive Ricci curvature, rigidity results were established (for instance in \cite{MR1134481,MR615628}) before the role of flows in the monotonicity of the functionals associated to the inequalities was clarified (see in particular \cite{MR2381156,MR2459454,dolbeault:hal-00784887}). However, such results are of interest by themselves.

\subsection{The case of a superlinear elliptic equation.}\label{Sec:Rigidity1}

With the notations of Section~\ref{flowplargerthan2}, consider the equation
\be{Eqn:Ell1}
-\,\L f+\lambda\,f=f^{p-1}\,, 
\ee
where $f:\R\rightarrow \R_+$ and $p>2$. Note that this equation is the Euler-Lagrange equation of the functional \eqref{atlanta5}.
If $u$ is such that $f=u^\beta$, then we notice that the equation can be rewritten as 
\[
\L u+(\beta-1)\,\frac{|u'|^2}u\,\nu-\frac\lambda\beta\,u+\frac{u^\kappa}{\beta}=0\,,
\]
with $\kappa=\beta\,(p-2)+1$. As in \cite{DEKL}, we notice that
\[
\izp{(\L u)\,u^\kappa}=-\,\kappa\,\izp{u^{\kappa-1}\,|u'|^2}\quad\mbox{and}\quad\izp{\frac{|u'|^2}u\,u^\kappa\,\nu}=\izp{u^{\kappa-1}\,|u'|^2}\,,
\]
so that
\[
\izp{\(\L u+\kappa\,\frac{|u'|^2}u\,\nu\)\,u^\kappa}=0\,.
\]
Hence, by \eqref{atlanta4} we know that
\begin{multline*}
0=\izp{\(\L u+\kappa\,\frac{|u'|^2}u\,\nu\)\,\(\L u+(\beta-1)\,\frac{|u'|^2}u\,\nu-\frac\lambda\beta\,u\)}\\
=\(\frac{2\,p}{p-2}-\lambda\,\frac{\kappa-1}\beta\)\izp{|u'|^2\,\nu}+\izp{\left|u''-\frac{p+2}{6-p}\,\frac{|u'|^2}u\right|^2\nu^2}\,.
\end{multline*}
This proves
\begin{thm}\label{cor:Rigidity1} Let $p\in(2,\infty), p \not= 6$. Assume that $f$ is a positive solution of the equation~\eqref{Eqn:Ell1}. 
If $\lambda<\frac{2\,p}{(p-2)^2}$, then $f$ is constant.
\end{thm}

\subsection{The case of a sublinear elliptic equation.}\label{Sec:Rigidity2}

With the notations of Section~\ref{flowplessthan2}, consider the equation
\be{Eqn:Ell2}
-\,\L f-\lambda\,f+f^{p-1}=0
\ee
for $f: (-1,1) \rightarrow \R_+$ and $p\in (1,2)$. Note that this equation is the Euler-Lagrange equation of the functional \eqref{atlanta6}. If $u$ is such that $f=u^\beta$, then we notice that the equation can be rewritten as 
\[
\L u+(\beta-1)\,\frac{|u'|^2}u\,\xi+\frac\lambda \beta\,u-\frac{ u^\kappa}{ \beta}=0\,,
\]
with $\kappa=\beta\,(p-2)+1$. Exactly as in Section~\ref{Sec:Rigidity1}, we notice that
\[
\izpp{(\L u)\,u^\kappa}=-\,\kappa\,\izpp{u^{\kappa-1}\,|u'|^2\,\xi}\quad\mbox{and}\quad\izpp{\frac{|u'|^2}u\,u^\kappa\,\xi}=\izpp{u^{\kappa-1}\,|u'|^2\,\xi}\,,
\]
so that
\[
\izpp{\(\L u+\kappa\,\frac{|u'|^2}u\,\xi\)\,u^\kappa}=0\,.
\]
Hence, we know that
\begin{multline*}
0=\izpp{\(\L u+\kappa\,\frac{|u'|^2}u\,\xi\)\,\(\L u+(\beta-1)\,\frac{|u'|^2}u\,\xi+\frac \lambda \beta\,u\)}\\
=\(\frac{2\,p}{2-p}+\lambda\,\frac{\kappa-1}\beta\) \izpp{|u'|^2\,\xi}+\izpp{\left|u''-\frac{p+2}{6-p}\,\frac{|u'|^2}u\right|^2\xi^2}\,.
\end{multline*}
This proves
\begin{thm} Let $ 1<p<2$ and $f$ a positive solution of the equation \eqref{Eqn:Ell2}. If $\lambda < \frac{2p}{(2-p)^2}$, then $f$ is constant.
\end{thm}

\section{Further considerations on flows} \label{Sec:Flows}

This section is devoted to the study of various flows associated with~\eqref{Ineq:duality1} and~\eqref{Ineq:duality2}. As we shall see below, fast diffusion flows with several different exponents are naturally associated with left-hand sides, while the heat flow appears as a gradient flow if we introduce an appropriate notion of distance in~\eqref{Ineq:duality1} for $p\in(2,3)$ and in~\eqref{Ineq:duality2} for $p\in(1,2)$.

\subsection{Fast diffusion flows.}\label{Sec:FDflows} The l.h.s.~in~\eqref{Ineq:duality1} is monotone increasing under the action of the flow associated to the fast diffusion flow
\be{FD}
\partial_tG=\sigma(t)\,\partial_y^2G^m+\partial_y(y\,G)\quad(t,y)\in\R^+\times\R\,,
\ee
where
\[
m=\frac{p+2}{3\,p-2}\;,
\]
and $\sigma(t)$ is adjusted at every $t\ge0$ so that $\frac d{dt}\iRy{G(t,y)\,|y|^2}=0$. The growth rate is determined by the Gagliardo-Nirenberg-Sobolev inequality
\be{GN}
\nrml u{2a}\le C(a)\,\nrml{u'}2^\theta\,\nrml u{a+1}^{1-\theta}\quad\mbox{with}\quad a:=\frac 1{2\,m-1}=\frac{3\,p-2}{6-p}\;,
\ee
which introduces the restriction
\[
2<p<6\quad\Longleftrightarrow\quad a\in(1,\infty)\quad\Longleftrightarrow\quad m\in\(\tfrac 12,1\).
\]
See \cite{1004,1104} for related considerations.

We can also use a more standard framework (see \cite{MR1940370,MR1777035}) as follows. For any $m\in(0,1)$, consider the usual fast diffusion equation in self-similar variables
\[
\partial_tG=\partial_y^2 G^m+\partial_y\cdot(y\,G)\quad(t,y)\in\R^+\times\R
\]
for some $m\in(0,1)$ and define the \emph{generalized entropy} by
\[
\mathcal F_1[G]:=\frac 1{m-1}\iRy{G^m}+\frac 12\iRy{G\,|y|^2}\,.
\]
The equation preserves the mass $M:=\iRy G$ and the entropy converges with an exponential rate towards its asymptotic value which is given by the Barenblatt profile
\[
G_\infty(y)=\(C+\tfrac{1-m}{2\,m}\,|y|^2\)^\frac1{m-1}\quad\forall\,y\in\R
\]
with same mass as the solution, \emph{i.e.}~with $C$ such that $\iRy{G_\infty}=M$. Since
\[
\mathcal F_1[G_\lambda]=\frac{\lambda^{m-1}}{m-1}\iRy{G^m}+\frac{\lambda^{-2}}2\iRy{G\,|y|^2}\quad\mbox{if}\quad G_\lambda(y):=\lambda\,G(\lambda\,y)\,,
\]
an optimization with respect to the parameter $\lambda>0$ shows that
\[
\mathcal F_1[G]\ge\mathcal F_1[G_\lambda]=\(\tfrac12-\tfrac1{1-m}\)\,\(\iRy{G^m}\)^\frac2{1+m}\(\iRy{G\,|y|^2}\)^{-\frac{1-m}{1+m}},
\]
which again shows that the l.h.s.~in~\eqref{Ineq:duality1} (raised to the appropriate exponent and multiplied by some well-defined constant) is the optimal value of $\mathcal F$.

Similarly, the l.h.s.~in~\eqref{Ineq:duality2} is monotone increasing under the action of the flow associated to the fast diffusion flow \eqref{FD} with
\[
m=\frac2{4-p}\;,
\]
and $\sigma(t)$ is again adjusted at every $t\ge0$ so that $\frac d{dt}\iRy{G(t,y)\,|y|^2}=0$. The growth rate is determined by \eqref{GN} with $a=\frac 1{2\,m-1}=\frac4p-1$, $p\in(1,2)$. Alternatively, we can also consider the entropy functional $\mathcal F$ as above.

\subsection{Gradient flows, entropies and distances.}\label{Sec:GradientFlows}

\subsubsection{Case \texorpdfstring{$p\in(1,2)$.}{p in(1,2)}}

Let us start with a simple computation based on the heat equation
\[
\partial_t\rho=\Delta\rho\quad x\in\R^d\,,\;t>0\,.
\]
Since the dimension plays no role, we can simply assume that $d\ge1$. Under appropriate assumptions on the initial datum, the mass $M$ of a non-negative solution is preserved along the evolution: $\frac d{dt}\int_{\R^d}\rho(t,x)\,dx=0$. A standard computation (see for instance \cite{pre05312043}) shows that
\be{Eqn:EntropyProduction}
\frac d{dt}\int_{\R^d}\rho^q\,dx=-\,4\,\frac{q-1}q\int_{\R^d}|\nabla\rho^{q/2}|^2\,dx\,.
\ee
With $f=\rho^{q/2}$, $p=2/q\in(1,2)$ and the Gagliardo-Nirenberg-Sobolev inequality
\[
\nrmR{\nabla f}2^\theta\,\nrmR fp^{1-\theta}\ge\mathsf C_{\mathrm GN}(p,d)\,\nrmR f2\quad\forall\,f\in\mathrm H^1(\R^d)\,,
\]
where $\theta=\frac{d\,(2-q)}{2\,d-q\,(d-2)}$, we find that
\[
\frac d{dt}\int_{\R^d}\rho^q\,dx\le-\,4\,\frac{q-1}q\,\(\frac{\mathsf C_{\mathrm GN}(p,d)}{M^{1-\theta}}\)^\frac1\theta\(\int_{\R^d}\rho^q\,dx\)^\frac1{2\,\theta}\,,
\]
which gives an explicit algebraic rate of decay of the entropy $\int_{\R^d}\rho^q\,dx$.

We will now introduce a notion of \emph{distance} as in \cite{MR2448650}, which is well-adapted to our setting. We refer to \cite{MR2448650} for a rigorous approach and consider the problem at formal level only. First of all one can consider the system
\[
\left\{\begin{array}{l}
\partial_t\rho+\nabla\cdot w=0\\[6pt]
\partial_tw=\Delta w
\end{array}\right.
\]
so that
\be{Eqn:EntropyFlow}
\frac d{dt}\int_{\R^d}\rho^q\,dx=-\,4\,\frac{q-1}q\int_{\R^d}\rho^{q-2}\,\nabla\rho\cdot w\,dx\,.
\ee
Let $\alpha=2-q$ and define the \emph{action functional} as
\[
\mathsf A_\alpha[\rho,w]:=\iRd{\frac{|w|^2}{\rho^\alpha}}\,.
\]
We recall that $\alpha\in(0,1)$ if and only if $q\in(1,2)$ or, equivalently $p=2/q\in(1,2)$. The above flow reduces to the heat flow if $w=-\nabla\rho$. If $\rho_0$ and $\rho_1$ are two probability densities, we can define a distance $d_\alpha$ between $\rho_0$ and $\rho_1$ by
\[
d_\alpha^2(\rho_0,\rho_1):=\inf\left\{\int_0^1\mathsf A_\alpha[\rho_s,w_s]\,ds\,:\,(\rho_s,w_s)\mbox{ is admissible}\right\}\,,
\]
where an admissible path connecting $\rho_0$ to $\rho_1$ is a pair $(\rho_s,w_s)$ parametrized by a coordinate $s$ ranging between $0$ and $1$, so that the endpoint densities are $\rho_{s=0}=\rho_0$ and $\rho_{s=1}=\rho_1$, $w_s$ is vector field and $(\rho_s,w_s)$ satisfies a continuity equation:
\[
\partial_s\rho_s+\nabla\cdot w_s=0\,.
\]
We can also define a notion of \emph{instant velocity} at point $s\in(0,1)$ along a path $(\rho_s)_{0\le s\le1}$ by
\[
\dot{|\rho_s|}^2:=\inf\left\{\mathsf A_\alpha[\rho_s,w]\,:\,\partial_s\rho_s+\nabla\cdot w=0\right\}\,.
\]

Consider now a given path $(\rho_t,w_t)_{t>0}$. Using~\eqref{Eqn:EntropyFlow} and a Cauchy-Schwarz inequality, we know that
\[
-\frac d{dt}\int_{\R^d}\rho_t^q\,dx\le q\,(q-1)\,\sqrt{\mathsf A_\alpha[\rho_t,\nabla\rho_t]\,\mathsf A_\alpha[\rho_t,w_t]}\,,
\]
so that
\[
-\frac d{dt}\int_{\R^d}\rho_t^q\,dx\le q\,(q-1)\,\sqrt{\mathsf A_\alpha[\rho_t,\nabla\rho_t]}\,\dot{|\rho_t|}\,,
\]
if the path is optimal for our notion of distance, \emph{i.e.}~$\dot{|\rho_t|}^2=\mathsf A_\alpha[\rho_t,w_t]$. On the other hand, $w_t=-\nabla\rho_t$ defines an admissible path along the heat flow and in that case we know from~\eqref{Eqn:EntropyProduction} that
\[
-\frac d{dt}\int_{\R^d}\rho_t^q\,dx=q\,(q-1)\,\mathsf A_\alpha[\rho_t,\nabla\rho_t]\,.
\]
If $(\rho_t)_{t>0}$ is the gradient flow of $\int_{\R^d}\rho^q\,dx$ with respect to $d_\alpha$, then on the one hand we have
\[
q\,(q-1)\,\mathsf A_\alpha[\rho_t,\nabla\rho_t]\le-\frac d{dt}\int_{\R^d}\rho_t^q\,dx\,,
\]
and on the other hand, using $w_t=-\nabla\rho_t$ as a test function in the definition of $\dot{|\rho_t|}$, we find that $\dot{|\rho_t|}^2\le\mathsf A_\alpha[\rho_t,\nabla\rho_t]$, thus showing that
\[
\dot{|\rho_t|}^2=\mathsf A_\alpha[\rho_t,\nabla\rho_t]\,.
\]
This is the desired result: the heat equation is the \emph{gradient flow} of $\int_{\R^d}\rho^q\,dx$ with respect to $d_\alpha$ if $q=2/p$ and $\alpha=2-q$.

\subsubsection{Case \texorpdfstring{$p>2$.}{p>2}}

One can consider the system
\[
\left\{\begin{array}{l}
\partial_t\rho+\nabla\cdot w=0\\[6pt]
\partial_tw=\Delta w
\end{array}\right.
\]
so that
\[
\frac d{dt}\int_{\R^d}\rho^\frac p2\,dx=-\,\frac14\,p\,(p-2)\int_{\R^d}\rho^{\frac p2-2}\,\nabla\rho\cdot w\,dx\,.
\]
Since
\[
\left|\;\int_{\R^d}\rho^{\frac p2-2}\,\nabla\rho\cdot w\,dx\;\right|^2\le \mathsf A_\alpha [\rho,w]\,\int_{\R^d}\rho^{p-3}\,|\nabla\rho|^2\,dx\,,
\]
with $\alpha=3-p$, it is rather straightforward to see that the equation
\[
\partial_t\rho=\Delta\rho^{2-\frac p2}
\]
is such that
\[
\frac d{dt}\int_{\R^d}\rho^\frac p2\,dx=-\,\frac18\,p\,(p-2)\,(4-p)\int_{\R^d}\frac{|\nabla\rho|^2}\rho\,dx\,,
\]
and hence can be interpreted as the gradient flow of $\rho\mapsto\int_{\R^d}\rho^\frac p2\,dx$ with respect to the distance~$d_\alpha$ and optimal descent direction given by $w=-\nabla\rho^{2-p/2}$ if $2<p<3$. Recall that conservation of mass holds only if $2-\frac p2>1-\frac 1d$, which is an additional restriction on the range of $p$.

\subsubsection{Comments}

The above gradient flow approaches are formal but can be fully justified. See \cite{Ambrosio-Gigli-Savare05} and \cite{MR2448650}. Difficulties lie in the fact that paths have to be defined on a space of measures (vector valued measures in case of $w$) and various regularizations are needed, as well as reparametrizations of the paths. This approach can also be carried out in self-similar variables (the heat equation has then to be replaced by a Fokker-Planck equation) and provides exponential rates of convergence in relative entropy with respect to the stationary solution, or with respect to the invariant measure if one works in the setting of the Ornstein-Uhlenbeck equation. The gradient flow structure of the equation with respect to some appropriate notion of distance has been studied in  \cite{pre05312043,doi:10.1137/110835190}  and the equivalent of McCann's condition for geodesic convexity of the corresponding functional has been established in~\cite{MR2565840}. 
 The precise connection of Gagliardo-Nirenberg-Sobolev inequalities with W.~Beckner's interpolation inequalities \cite{Beckner89,pre05312043} in case of a Gaussian measure and M.~Agueh's computations in \cite{MR2263417,MR2427077} is still to be done.

As a final remark in this section, let us observe that it is crucial for our approach that the action functional $(\rho,w)\mapsto\mathsf A_\alpha[\rho,w]$ is convex. An elementary computation shows that this implies that $\alpha$ is in the interval $[0,1]$, where for $\alpha=1$ (that is, $q=1$ and $p=2$), the distance $d_1$ corresponds to the usual Wasserstein distance according to the Benamou-Brenier characterization in \cite{Benamou-Brenier00}, while for $\alpha=0$ (that is, $q=2$ and $p=1$), the distance $d_0$ corresponds to the usual $\mathrm H^{-1}$ notion of distance. If we now consider the case $p>3$, the functional $\mathsf A_\alpha$ is no longer convex and, although at a formal level the computations are still the same, it is no longer possible to define a meaningful notion of distance. It is therefore an open question to understand whether there is a notion of gradient flow which is naturally associated to the Gagliardo-Nirenberg-Sobolev inequalities with $p>3$ or not.


\section{Concluding remarks}\label{Sec:Conclusion}

Well-chosen \emph{entropy} functionals are exponentially decreasing under the action of the flow defined by the fast diffusion equation and the optimal rate of decay is given by the best constant in a special family of Gagliardo-Nirenberg inequalities: see for instance \cite{MR1777035,MR1940370,MR2126633,BBDGV,BDGV}. Moreover, self-similar solutions, the so-called Barenblatt functions, are extremal for the inequalities (see \cite{MR1940370,Gunson91}). An explanation for this fact has been given in \cite{MR1842429} by F.~Otto: the fast diffusion equation is the gradient flow of the entropy with respect to the Wasserstein distance while the entropy (at least in some range of the exponent) is displacement convex. This has been exploited by D.~Cordero-Erausquin, B.~Nazaret and C.~Villani in~\cite{MR1842429} in order to provide a proof of the Gagliardo-Nirenberg inequalities associated with fast diffusion using mass transportation techniques. Such a method heavily relies on the explicit knowledge of the Barenblatt functions, as well as the reformulation that was given in \cite{MR2053603}. A striking point of the method of \cite{MR1842429} is a nice duality which relates the Gagliardo-Nirenberg inequalities with a much simpler expression, which again has the Barenblatt functions as optimal functions. 

Not so many interpolation inequalities have explicitly known optimal functions. Among Gagliardo-Nirenberg inequalities, the other well-known families are Nash's inequalities and the family which corresponds to the one-dimensional case. This was observed long ago and M.~Agueh has investigated in \cite{MR2263417,MR2427077} how Barenblatt functions are transformed into optimal functions for the inequalities. We refer to these two papers for an expression of the explicit transport map $\varphi$ in case of optimal functions. In this paper we have focused our attention on the one-dimensional Gagliardo-Nirenberg inequalities and established \emph{duality results} which are analogues to the ones in~\cite{MR1842429}: see Section~\ref{Sec:Duality}. A remarkable fact is that the dual functional is associated in both cases with an entropy corresponding to a fast diffusion equation.

In \cite{pre05312043,MR2448650,doi:10.1137/110835190}, some interpolation inequalities associated with $p<2$ have been studied. We have adapted the methods that can be found there to establish that for some appropriate notion of \emph{distance,} which is not anymore the Wasserstein distance, a notion of \emph{gradient flow} is associated with the Gagliardo-Nirenberg-Sobolev inequalities.

\medskip Now let us summarize some aspects of the present paper before listing intriguing issues concerning flows. We have studied the Gagliardo-Nirenberg-Sobolev interpolation inequalities \eqref{GN1} and \eqref{GN2} using the three strategies mentioned in the introduction:
\begin{enumerate}
\item[(a)] The direct variational approach has been carried out in Appendix~\ref{Sec:AppendixB}, for completeness.
\item[(b)] The flow method has been studied in Section~\ref{Sec:monotonicity}, and summarized in Theorems~\ref{thm:flow1} and~\ref{thm:flow2}; the corresponding rigidity results are stated in Section~\ref{Sec:Rigidity}.
\item[(c)] The duality by mass transportation is the subject of Section~\ref{Sec:Duality}. Optimality has been checked in Section~\ref{Sec:Optimality}.
\end{enumerate}

There is a natural notion of \emph{flow associated with the dual problem} obtained by mass transportation, which is of fast diffusion type; this flow can be seen as a gradient flow with respect to Wasserstein's distance. There is a also notion of \emph{gradient flow} for a well-chosen notion of distance (which is not, in general, Wasserstein's distance), that is studied in Section~\ref{Sec:GradientFlows} and for which optimal rates of decay are given by our Gagliardo-Nirenberg-Sobolev inequalities, but the connection with the mass transportation of Section~\ref{Sec:Duality} is still to be clarified.

Method (b) is in a sense surprising. We select a special optimal function and exhibit another \emph{nonlinear diffusion flow}, which is not translation invariant, that forces the solution with any initial datum to converge for large times to the special optimal function we have chosen. The non-negativity of the associated functional is equivalent to the Gagliardo-Nirenberg-Sobolev inequality and the striking property of the flow is that our functional is monotonously non-increasing. The functional is invariant under translations, and any solution corresponding to a translation of the optimal function returns to the initially chosen optimal function, keeping the functional at its minimal level. This is explained by conformal invariance on the sphere and is anything but trivial. This phenomenon, namely that the functional is invariant under translations (which is the same as conformal invariance in other variables) but nevertheless non-increasing under the flow that converges to a single function is at the heart of the competing symmetry approach by E.~Carlen and M.~Loss in \cite{MR1038450}. How this last flow is connected with the other ones is also an open question. At least the computation that shows why the functional decays along the flow clarifies a bunch of existing computations for proving rigidity results for nonlinear elliptic equations written on $d$-dimensional spheres and for the ultraspherical operator.

\appendix\section{Two useful identities}\label{Sec:AppendixA}

On the real interval $\Omega$, let us consider the measure $\nub=\nu^\bb\,dx$ for some positive function $\nu$ on $\Omega$. We consider the space $\mathrm L^2(\Omega,\nub)$ endowed with the scalar product
\[
\scal{f_1}{f_2}=\ixb{f_1\,f_2}\,.
\]
On $\mathrm L^2(\Omega,\nub)$, we define the self-adjoint operator
\[
\Lab f:=\nu^\aa\,f''+\frac{\aa+\bb}\aa\,(\nu^\aa)'\,f'
\]
which satisfies the identity
\[
\scal{f_1}{\Lab f_2}=-\ixb{f_1'\,f_2'\,\nu^\aa}\;.
\]
This identity determines the domain of $\Lab$. We will now establish two useful identities.
\begin{prop}\label{Prop:2id} Assume that $\aa$ and $\bb$ are two reals numbers with $\aa\neq0$ and consider a smooth positive function $u$ which is compactly supported in $\Omega$. With the above notations, we have
\[
\ixb{(\Lab u)^2}=\ixmu{|u''|^2}{2\aa+\bb}-\frac{\aa+\bb}\aa\,\ixb{\nu^\aa\,(\nu^\aa)''\,|u'|^2}
\]
and
\[
\ixb{(\Lab u)\,\frac{|u'|^2}u\,\nu^\aa}=\frac{\aa+\bb}{2\,\aa+\bb}\ixb{\frac{|u'|^4}{u^2}\,\nu^{2\aa}}-\frac{\aa+2\,\bb}{2\,\aa+\bb}\ixb{u''\,\frac{|u'|^2}u\,\nu^{2\aa}}\,.
\]
\end{prop}
\begin{proof} As a preliminary observation, we can observe that
\[
\left[\frac d{dx},\Lab\right]f=(\nu^\aa)'\,f''+\frac{\aa+\bb}\aa\,(\nu^\aa)''\,f'\,,
\]
so that we immediately get
\begin{eqnarray*}
\ixb{(\Lab u)^2}&=&-\ixb{\nu^\aa\,u'\,(\Lab u)'}\\
&=&-\ixb{\nu^\aa\,u'\,(\Lab u')}-\ixb{\nu^\aa\,u'\,\left[(\nu^\aa)'\,u''+\frac{\aa+\bb}\aa\,(\nu^\aa)''\,u'\right]}\\
&=&\ixb{(\nu^\aa\,u')'\,\nu^\aa\,u''}-\ixb{\nu^\aa\,u'\,\left[(\nu^\aa)'\,u''+\frac{\aa+\bb}\aa\,(\nu^\aa)''\,u'\right]}\\
&=&\ixmu{|u''|^2}{2\aa+\bb}-\frac{\aa+\bb}\aa\,\ixb{\nu^\aa\,(\nu^\aa)''\,|u'|^2}\,.
\end{eqnarray*}
On the other hand, using an integration by parts, we notice that
\begin{multline*}
\ixb{u''\,\frac{|u'|^2}u\,\nu^{2\aa}}=\frac 13\ixb{(|u'|^2\,u')'\,\frac{\nu^{2\aa}}u}\\
=\frac 13\ixb{\frac{|u'|^4}{u^2}\,\nu^{2\aa}}-\frac{2\,\aa+\bb}{3\,\aa}\ixb{\frac{|u'|^2\,u'}u\,(\nu^{\aa})'\,\nu^{\aa}}\,,
\end{multline*}
thus proving that
\[
\ixb{\frac{|u'|^2\,u'}u\,(\nu^{\aa})'\,\nu^{\aa}}=-\frac{3\,\aa}{2\,\aa+\bb}\ixb{u''\,\frac{|u'|^2}u\,\nu^{2\aa}}+\frac\aa{2\,\aa+\bb}\ixb{\frac{|u'|^4}{u^2}\,\nu^{2\aa}}\,.
\]
Using the definition of $\Lab$, we have
\[
\ixb{(\Lab u)\,\frac{|u'|^2}u\,\nu^\aa}=\ixb{\(\nu^\aa\,u''+\frac{\aa+\bb}\aa\,(\nu^\aa)'\,u'\)\,\frac{|u'|^2}u\,\nu^\aa}\,,
\]
thus concluding the proof.\end{proof}

From a practical point of view, we will apply Proposition~\ref{Prop:2id} either to $\Omega=(-1,1)$ and $\nu(x):=1-x^2$, or to $\Omega=\R$ and $\nu(x):=1+x^2$.

\section{The direct variational approach}\label{Sec:AppendixB}

For completeness, let us give a statement on optimality in \eqref{GN1} and \eqref{GN2} according to the approach (a) of the introduction. Let us start with the case $p\in(2,\infty)$. We recall that the inequality \eqref{GN1} can be written as
\[
\eqref{GN1}\hspace*{2cm}\nrml fp\le\mathsf C_{\rm GN}(p)\,\nrml{f'}2^\theta\,\nrml f2^{1-\theta}\quad\forall\,f\in\H^1(\R)\,.\hspace*{2cm}
\]
with $\theta=\frac{p-2}{2\,p}$. By standard results of the concentration-compactness method (see for instance \cite{778970,778974}), there exists an optimal function $f$ for \eqref{GN1}. Because of the homogeneity, $\nrml fp$ can be chosen arbitrarily and then, up to a scaling, it is straightforward to check that $f$ can be chosen in order to solve
\be{EL1}
-\,(p-2)^2\,f''+4\,f-\,2\,p\,|f|^{p-2}\,f=0\,.
\ee
A special solution is given by
\[
f_\star(x)=\frac 1{(\cosh x)^\frac2{p-2}}\quad\forall\,x\in\R\,.
\]
\begin{prop}\label{Prop:MinGN1} Assume that $p\in(2,\infty)$. For any optimal function $f$ in \eqref{GN1}, there exists $(\lambda,\mu,x_0)\in\R\times(0,\infty)\times\R$ such that
\[
f(x)=\lambda\,f_\star\big(\mu\,(x-x_0)\big)\quad\forall\,x\in\R\,.
\]
\end{prop}
\begin{proof} Because of the scaling invariance and the homogeneity in \eqref{GN1}, it is enough to prove that $f_\star$ is the unique solution of~\eqref{EL1}. Since $f\in\H^1(\R)$, we also know that $f$ and $f'$ are exponentially decaying as $|x|\to+\infty$. By multiplying \eqref{EL1} by $f'$ and integrating from $-\infty$ to $x$, we find that
\[
\mathrm E[f]=\frac 12\,(p-2)^2\,|f'|^2+2\,|f|^2-\,2\,|f|^p
\]
does not depend on $x$. On the other hand, taking into account the limits as $|x|\to+\infty$, we know that $\mathrm E[f]=0$. Let $x_0\in\R$ be such that $|f(x_0)|=\max_\R|f|$. Up to translation, we may assume that $x_0=0$, so that $f'(0)=0$ and $0=\mathrm E[f]=2\(|f(0)|^2-\,|f(0)|^p\)$, thus proving that $f(0)=\pm1$. By the Cauchy-Lipschitz theorem, there exists therefore a unique solution $f$ to~\eqref{EL1} which attains its maximum at $x=0$ and hence we get that $f=f_\star$.\end{proof}

Now let us consider the case $p\in(1,2)$ and turn our attention to \eqref{GN2}. We recall that the inequality \eqref{GN2} can be written as
\[
\eqref{GN2}\hspace*{2cm}\nrml f2\le\mathsf C_{\rm GN}(p)\,\nrml{f'}2^\eta\,\nrml fp^{1-\eta}\quad\forall\,f\in\H^1(\R)\,,\hspace*{2cm}
\]
with $\eta=\frac{2-p}{2+p}$. By standard results of the concentration-compactness method again, there exists an optimal function $f$ for \eqref{GN2}. Because of the homogeneity, $\nrml fp$ can be chosen arbitrarily and then, up to a scaling, it is straightforward to check that $f$ can be chosen in order to solve
\be{EL2}
-\,(2-p)^2\,f''-4\,f+\,2\,p\,|f|^{p-2}\,f=0\,.
\ee
A special solution is given by
\[
f_*(x)=(\cos x)^\frac2{2-p}\quad\forall\,x\in\left[-\tfrac\pi2,\tfrac\pi2\right]\,,\quad f_*(x)=0\quad\forall\,x\in\R\setminus\left[-\tfrac\pi2,\tfrac\pi2\right]\,.
\]
Moreover by the \emph{compact support principle} (see \cite{MR0390473} and \cite{MR1715341,MR768629} for more recent developments), we know that any solution of \eqref{EL2} in $\H^1(\R)$ has compact support.
\begin{prop}\label{Prop:MinGN2} Assume that $p\in(1,2)$. For any optimal function $f$ in \eqref{GN2}, there exists $(\lambda,\mu,x_0)\in\R\times(0,\infty)\times\R$ such that
\[
f(x)=\lambda\,f_*\big(\mu\,(x-x_0)\big)\quad\forall\,x\in\R\,.
\]
\end{prop}
\begin{proof} Because of the homogeneity and of the scale invariance, finding an optimal function for \eqref{GN2} is equivalent to minimizing the functional
\[
f\mapsto\mathcal G[f]:=\iR{|f'|^2}+\iR{|f|^p}-\mathsf C\(\iR{|f|^2}\)^\frac{p+2}{6-p}
\]
for some appropriately chosen positive constant $\mathsf C$. A unique value of $\mathsf C$ can indeed be found and computed in terms of $\mathsf C_{\rm GN}(p)$ so that the minimum of $\mathcal G$ is achieved and equal to $0$. Let $f$ be the minimizer and assume that $f=\sum_{i\ge1} f_i$ where $(f_i)_{i\ge1}$ is a family of functions with disjoint compact supports made of bounded intervals. Assume that the number of intervals is larger than $1$. Since $\frac{p+2}{6-p}<1$, by concavity we get that
\[ 
\sum_{i\ge1}\mathcal G[f_i]<\mathcal G[f]=0 
\]
a contradiction. This proves that the support of $f$ is made of a single interval. Then the proof goes as in the case $p>2$. By considering $\mathrm E[f]=\frac 12\,(2-p)^2\,|f'|^2-2\,|f|^2+\,2\,|f|^p$ which again does not depend on $x$, we get that at its maximum (assumed to be achieved at $x=0$), we have $f(0)=\pm1$ and conclude again using a uniqueness argument deduced from the Cauchy-Lipschitz theorem that $f=f_*$.\end{proof}

\bigskip\noindent{\sl\small \copyright~2013 by the authors. This paper may be reproduced, in its entirety, for non-commercial purposes.}
\bibliographystyle{siam}
\bibliography{References}
\affiliationone{
J.~Dolbeault and M.J.~Esteban\\ Ceremade, Universit\'e Paris-Dauphine, Place de Lattre de Tassigny,\newline 75775 Paris C\'edex~16, France
\email{\Email{dolbeaul@ceremade.dauphine.fr}\\ \Email{esteban@ceremade.dauphine.fr}}}
\affiliationtwo{
A.~Laptev\\ Department of Mathematics, Imperial College London, Huxley Building, 180 Queen's Gate, London SW7 2AZ, UK\email{\Email{a.laptev@imperial.ac.uk}}}
\affiliationthree{M.~Loss\\ School of Mathematics, Skiles Building, Georgia Institute of Technology, Atlanta GA 30332-0160, USA
\email{\Email{loss@math.gatech.edu}}}
\end{document}